\documentclass[reqno,11pt]{amsproc}

\usepackage{xcolor}
\definecolor{myurlcolor}{rgb}{0,0,0.4}
\definecolor{mycitecolor}{rgb}{0,0.5,0}
\definecolor{myrefcolor}{rgb}{0.5,0,0}
\usepackage[breaklinks=true]{hyperref}
\hypersetup{colorlinks,
linkcolor=myrefcolor,
citecolor=mycitecolor,
urlcolor=myurlcolor}

\usepackage{amssymb}
\usepackage{amsmath}
\usepackage{amsfonts}
\usepackage{microtype}
\usepackage{xcolor}
\usepackage[left=2.6cm,right=2.6cm]{geometry}
\usepackage{mathrsfs}
\usepackage{mathtools}
\usepackage{amsthm}

\usepackage{thmtools} 
\usepackage[capitalize]{cleveref}
\Crefname{figure}{Figure}{Figures}
\usepackage{listings}

\usepackage[backend=biber,style=ieee,sorting=none,backref=true,citestyle=numeric-comp,giveninits=false]{biblatex}
\usepackage{tikz,tikz-cd}
\usetikzlibrary{arrows,cd,decorations.markings,arrows.meta,decorations.pathmorphing,shapes.geometric,calc}

\tikzset{->-/.style={decoration={markings,mark=at position .5 with {\arrow{>}}},postaction={decorate}}}
\tikzset{-<-/.style={decoration={markings,mark=at position .5 with {\arrow{<}}},postaction={decorate}}}
\tikzstyle{box}=[fill=white, draw=black, shape=rectangle, inner sep=14pt]
\tikzstyle{forward arrow}=[->-]
\tikzstyle{backward arrow}=[-<-]

\newcommand{\beq}{\begin{equation}}
\newcommand{\eeq}{\end{equation}}

\newcommand{\R}{\mathbb{R}}

\newcommand{\eps}{\varepsilon}
\newcommand{\F}{\mathcal{F}}	
\newcommand{\E}{\mathbb{E}}	

\newcommand{\ii}{\mathrm{i}}

\newtheorem{dummy}{Dummy}[section]	
\newtheorem{thm}[dummy]{Theorem}
\newtheorem{defn}[dummy]{Definition}

\newtheorem{lem}[dummy]{Lemma}

\theoremstyle{definition}
\newtheorem{rem}[dummy]{Remark}

\Crefname{thm}{Theorem}{Theorems}
\Crefname{prob}{Problem}{Problems}

\allowdisplaybreaks

\let\originalleft\left
\let\originalright\right
\renewcommand{\left}{\mathopen{}\mathclose\bgroup\originalleft}
\renewcommand{\right}{\aftergroup\egroup\originalright}

\usepackage{enumitem}
\setlist[enumerate]{label=(\roman*),itemsep=5pt,topsep=8pt}
\setlist[itemize]{label=$\triangleright$,itemsep=5pt,topsep=6pt}
\Crefformat{enumi}{#2#1#3}

\newcommand{\newterm}[1]{\textbf{#1}}

\usepackage{hyphenat}

\usepackage{todonotes}

\begin{document}



\title{Dynamics from iterated averaging}

\author{Tobias Fritz}
\address{Department of Mathematics, University of Innsbruck, Austria}
\email{tobias.fritz@uibk.ac.at}

\author{Nicol\'as Rivera}
\address{Institute of Statistics, University of Valparaíso, Chile}
\email{nicolas.rivera@uv.cl}



\thanks{\textit{Acknowledgements.} We thank Paolo Perrone for raising the question which led to this work, Omer Tamuz and puzzling.SE users \emph{Pranay} and \emph{Nitrodon} for discussion and Jakub Czartowski for further comments including pointers to the literature on quantum thermodynamics. Nicol\'as Rivera was supported by ANID FONDECYT grant No 11251301 and ANID SIA grant No 85220033.}

\begin{abstract}
	We prove that for a standard Lebesgue space $X$, the strong operator closure of the semigroup generated by conditional expectations on $L^\infty(X)$ contains the group of measure-preserving automorphisms.
	This is based on a solution to the following puzzle: given $n$ full water tanks, each containing one unit of water, and $n$ empty ones, how much water can be transferred from the full tanks to the empty ones by repeatedly equilibrating the water levels between pairs of tanks?	
\end{abstract}

\maketitle

\section{Introduction}
\label{introduction}

Consider $n$ full water tanks, each containing one unit of water, and $n$ empty ones.
The tanks are all identical and at the exact same elevation.
Suppose that you have a hose at your disposal, using which you can transfer water from one tank to another.
But without a pump, all that you can do is to \emph{equilibrate} the water levels between two tanks.
Then under these conditions, how much water can you transfer from the full tanks to the empty ones?

This problem has a surprising answer: for large $n$, it is possible to transfer \emph{almost all} water from the full tanks to the empty ones.

\begin{thm}\label{thm:optimaltransport}
	\label{tanks}
	For this problem, there is a strategy which ends up with only
	\[
		\frac{n}{4^n} \binom{2n}{n} = \sqrt{\frac{n}{\pi}} + O(n^{-{1/2}})
	\]
	units of water in the originally full tanks, and this is optimal.
\end{thm}

\begin{rem}
	We present the proof in~\cref{optimal_strategies}, but already point out now that most of this result is not due to us.
	First of all, the strategy is a straightforward discretization of the well-known \emph{countercurrent heat exchanger} method for transferring heat between two fluids (\cref{heat}).
	Moreover, the strategy and its asymptotics also appear in recent work by Czartowski, de Oliveira Junior and Korzekwa in the context of quantum thermodynamics~\cite[Lemma~1]{COK}, but without the exact expression or an optimality proof.
	Finally, the optimality proof is largely due to user \emph{Nitrodon} on puzzling.stackexchange.com~\cite{nitrodon}.
	It applies to any initial configuration of water levels.
\end{rem}

Curiously, the fraction of water which remains in the originally full tanks is precisely the probability for $2n$ tosses of a fair coin to yield exactly $n$ heads, namely $4^{-n} \binom{2n}{n}$.
The fact that we can transfer all water as $n \to \infty$ may seem surprising: equilibrating the water between two tanks is a kind of averaging operation, and averaging should result in a general tendency towards equalization.
Instead, the theorem asserts that there is a strategy which amounts to a \emph{permutation} of the full tanks with the empty ones---at least to a very good approximation as $n \gg 1$.

An equilibration of two tanks can equivalently be described as taking the conditional expectation of the function which assigns to each tank its water level with respect to the $\sigma$-algebra which identifies these two tanks while distinguishes all others.
Thus a closely related question to the water tank puzzle is: which functions can be obtained by repeatedly taking conditional expectations of a function on a probability space, say a standard Lebesgue space?
Here the answer is similarly surprising: such repeated averaging can approximate any dynamics, by which we mean the action of a measure-preserving automorphism on the function.
Moreover, this works for any finite number of functions at a time.

\begin{thm}
	\label{dynamics}
	Let $(X, \Sigma, \mu)$ be a standard Lebesgue probability space and $T : X \to X$ a measure-preserving automorphism.
	Then for all $f_1, \dots, f_n \in L^\infty(X)$ and $\eps > 0$, there is a finite sequence of sub-$\sigma$-algebras $\F_1, \ldots, \F_m \subseteq \Sigma$ such that
	\[
		\left\| f_i \circ T - \E_m \cdots \E_1 f_i \right\|_\infty < \eps \qquad \forall i = 1, \dots, n,
	\]
	where each $\E_j : L^\infty(X) \to L^\infty(X)$ is the conditional expectation operator with respect to $\F_j$.
\end{thm}

In other words, if one equips the space of bounded operators $L^\infty(X) \to L^\infty(X)$ with the strong operator topology, then the closure of the semigroup generated by conditional expectations contains the group of measure-preserving automorphisms.
Note that it does not matter whether we take $L^\infty(X)$ to consist of the real-valued or the complex-valued bounded measurable functions.

In the following two sections, we prove these results and give some additional context.

\begin{rem}
	\label{story}
	For us, the equilibration puzzle arose from the following question asked by Paolo Perrone.
	Given $f, g \in L^\infty(X)$ on a standard probability space $X$, write $f \ge g$ if $g$ can be written as a conditional expectation of $f$ with respect to some sub-$\sigma$-algebra.
	Then is the transitive closure of this relation an interesting new preorder?

	Our answer is negative: In order for it to be interesting, one will also at least want to take the topological closure in $L^\infty$.
	But then \cref{dynamics} forces $f \ge f \circ T \ge f$ for any measure-preserving automorphism $T : X \to X$, and thus one is forced to identify all functions related by such automorphisms.
	The resulting equivalence classes are precisely the compactly supported distributions on $\R$, and the preorder reduces to the usual Choquet order on distributions~\cite[Theorem~1.3.6]{winkler}.
	We discovered the water tank puzzle and its solution in the process of proving \cref{dynamics}.
\end{rem}

\section{Optimal equilibration strategies and proof of \texorpdfstring{\cref{tanks}}{Theorem 1.1}}
\label{optimal_strategies}

\lstdefinestyle{mystyle}{
    backgroundcolor=\color{white},   
    basicstyle=\ttfamily\small,
    breakatwhitespace=false,         
    breaklines=true,                 
    captionpos=b,                    
    keepspaces=true,                 
    numbers=left,                    
    numbersep=5pt,                  
    showspaces=false,                
    showstringspaces=false,
    showtabs=false,                  
    tabsize=2
}
\lstset{style=mystyle}

\begin{figure}
		\begin{lstlisting}
		For i in red tanks:
			For j in blue tanks:
				Equilibrate tanks i and j\end{lstlisting}
	\caption{Pseudocode for the first water tank strategy.}
	\label{pseudocode1}
\end{figure}

One interesting strategy\footnote{What we mean by a \emph{strategy} is simply any finite sequence of pairs of tanks, thought of as specifying a sequence of pairwise equilibrations.} to transfer water from the full tanks to the empty ones is given by the pseudocode of \cref{pseudocode1}.
To simplify the terminology, we call the originally full tanks \newterm{red} and the originally empty ones \newterm{blue}; this choice of colours hints at the following interpretation in terms of heat exchange.

\begin{figure}
\begin{tikzpicture}

\def\pipewidth{1.6cm}
\def\pipeheight{7cm}

\shade[top color=red,bottom color=red!30!blue] 
  (-\pipewidth,0) rectangle (0,\pipeheight);
\draw[very thick] (-\pipewidth,0) rectangle (0,\pipeheight);

\shade[bottom color=blue,top color=blue!30!red] 
  (0,0) rectangle (\pipewidth,\pipeheight);
\draw[very thick] (0,0) rectangle (\pipewidth,\pipeheight);

\foreach \y in {0, 1.5, 3, 4.5, 6} {
  \draw[->, thick, green!30!gray] ($ (\pipewidth/2 - 0.2cm, \y + 0.0) $) -- ($ (\pipewidth/2 - 0.2cm, \y + 1.0) $);
  \draw[->, thick, green!30!gray] ($ (\pipewidth/2 + 0.3cm, \y + 0.0) $) -- ($ (\pipewidth/2 + 0.3cm, \y + 1.0) $);
  \draw[->, thick, green!30!gray] ($ (-\pipewidth/2 - 0.3cm, \y + 1.0) $) -- ($ (-\pipewidth/2 - 0.3cm, \y + 0.0) $);
  \draw[->, thick, green!30!gray] ($ (-\pipewidth/2 + 0.2cm, \y + 1.0) $) -- ($ (-\pipewidth/2 + 0.2cm, \y + 0.0) $);
  \draw[very thick, decorate, yellow!30!gray, decoration={
    zigzag, segment length=4.9, amplitude=.9}] ($ (-\pipewidth/5, \y + 0.5) $) -- ($ (\pipewidth/5, \y + 0.5) $);
}

\node[anchor=south] at (-\pipewidth/2,\pipeheight) {very hot};
\node[anchor=north] at (-\pipewidth/2,0) {cold};
\node[anchor=south] at (\pipewidth/2,\pipeheight) {\phantom{y\!\!\!}hot};
\node[anchor=north] at (\pipewidth/2,0) {very cold};

\end{tikzpicture}
\caption{Schematic of a countercurrent heat exchanger. Hot fluid enters in the left tube on top and cold fluid in the right tube at the bottom. Where the tubes touch, heat is exchanged (squiggly lines). As each fluid exits its tube, its temperature is close to the temperature of the other fluid as it enters, and matches it exactly in the ideal case.}
\label{heat_exchanger}
\end{figure}

\begin{rem}
	\label{heat}
	The strategy of \cref{pseudocode1} is closely related to \newterm{countercurrent heat exchangers}, which are devices for transferring heat from one fluid to another by running them in opposite directions through tubes exchanging heat (\cref{heat_exchanger}).

	More precisely, suppose that each fluid comes in $n$ discrete chunks.
	These chunks move along the tubes in steps that are discrete in both time and space; assume that the moves of the left chunks alternate with those of the right chunks and that 
	the tubes themselves are long enough to hold all chunks at once.
	At each step, the chunk in the left tube equilibrates with the chunk in the right tube at the same height, and then they both move on.
	Then what happens in this heat exchanger is equivalent to the for loops of \cref{pseudocode1}.
	The final two equilibration steps correspond to each fluid equilibrating on its own after all chunks have moved through the tubes.
\end{rem}


We can also ask how much water can be transferred from red to blue tanks starting with any initial configuration of water levels.
A non-deterministic strategy for this more general problem is given by the pseudocode of \cref{pseudocodeg}.
In the case where all red tanks start full and all blue ones empty, the strategy of \cref{pseudocode1} is an instance of this, as are all of the alternative protocols considered in~\cite[Appendix~F]{COK}.

The following optimality result is due to user \emph{Nitrodon} on puzzling.stackexchange.com, and we present a corrected version of their proof.

\begin{thm}[{\cite{nitrodon}}]\label{thm:nitrodon}
	Any strategy following the prescription of \cref{pseudocodeg} is optimal: no other finite sequence of equilibrations transfers more water from the red tanks to the blue ones.
\end{thm}

\begin{figure}
	\begin{lstlisting}
		Initialization:
			sort all tanks by decreasing water level from left to right;
			break ties such that red tanks are to the
			right of blue tanks of the same level

		While there is a red tank left of a blue tank:
			Equilibrate any red tank directly left of any blue tank \end{lstlisting}
	\caption{Pseudocode for a non-deterministic generalization of the two for loops of \cref{pseudocode1} to arbitrary initial water levels.}
	\label{pseudocodeg}
\end{figure}

\begin{proof}
	We write $x_a$ for the water level in each tank $a$.
	For the following arguments, we will use the fact that any strategy can act on any initial configuration of water levels, and its overall effect is linear in the $x_a$, since the strategy is represented by a bistochastic matrix. 
	Moreover, the effect of a single equilibration step is precisely the expectation value of randomly doing nothing or swapping the contents of the two tanks involved.
	Hence doing one of these two things instead of that equilibration step, while doing everything else the same, will result in a ``strategy with permutations'' that is at least as good as the original strategy.\footnote{More formally, what we mean by a strategy with permutations is that each step is either an equilibration of two tanks or a swap of their contents. Clearly strategies with permutations are much more powerful than strategies consisting of only equilibrations.}

	We now prove the claim through several intermediate steps.
	First, we argue that given any strategy $S$, there is a strategy $S'$ which transfers at least as much water to blue tanks as $S$ and additionally satisfies the following:
	\begin{enumerate}[label=(\roman*)]
		\item\label{firsti} Every step is of the following form: an equilibration of a red tank $a$ with a blue tank $b$ where $x_a > x_b$.
		\item\label{firstii} No pair of tanks is equilibrated twice.
		\item\label{firstiii} If a step equilibrates tanks $a$ and $b$, then there is no other tank whose water level lies strictly between those of $a$ and $b$ at that time.
	\end{enumerate}
	For the proof of \ref{firsti} and~\ref{firstii}, let $S'$ be any strategy that is at least as good as $S$ but as short as possible.
	We prove that $S'$ satisfies~\ref{firsti} and~\ref{firstii}.
	So suppose that~\ref{firsti} is violated for $S'$, and consider first the case that this is because the strategy equilibrates two tanks of the same color.
	Let $a$ and $b$ be the last pair of tanks of the same color which are equilibrated.
	By assumption on $S'$, removing this step results in a strictly worse strategy.
	Hence by what we observed in the previous paragraph, swapping $x_a$ and $x_b$ instead of equilibrating them results in a ``strategy with permutations'' that performs strictly better than $S'$.
	But instead of this swap, we achieve an equivalent effect by doing all subsequent equilibrations with $a$ instead with $b$ and vice versa.
	Since this would result in a strategy $S''$ that is strictly better than $S'$ and of the same length, this contradicts the assumptions on $S'$.
	Next, consider the case that $S'$ only equilibrates tanks of different colors, but that there is a step where a red tank $a$ is equilibrated with a blue tank $b$ where $x_a \le x_b$, and let us focus on the last time this happens.
	Then simply omitting this step results in a shorter strategy that is at least as good (given that all later steps transfer at least as much water from red to blue as before).
	Therefore we can assume that \ref{firsti} holds.

	Now suppose that~\ref{firstii} is violated for $S'$.
	Let tanks $a$ and $b$ be equilibrated at times $t_1$ and $t_2$ with $t_1 < t_2$, and such that no step between $t_1$ and $t_2$ uses both $a$ and $b$.
	If neither $a$ nor $b$ is used at any point between $t_1$ and $t_2$, then the second equilibration would be redundant and $S'$ would not be as short as possible.
	Thus at least one of them is used in this time interval.
	Now for all equilibrations with $a$ in this interval, do them with $b$ instead, and vice versa.
	Then the resulting new strategy $S''$ achieves the same outcome as $S$ in the same number of steps but violates~\ref{firsti}, which we already showed to be impossible.

	By~\ref{firsti} and~\ref{firstii}, we can already conclude that an optimal strategy exists, because there are only finitely many strategies satisfying these conditions.
	Thus from now on, we assume that $S'$ is an optimal strategy, and has the smallest number of steps among all such.
	For~\ref{firstiii}, we show that there is another strategy $S''$ that either obeys it or has its first violation at a later step than $S'$.
	Repeating this process results in a strategy that obeys~\ref{firstiii} as desired.

	Suppose the last violation of \ref{firstiii} in $S'$ occurs at a step where tanks $a$ and $b$ are equilibrated and the water level of tank $c$ is strictly between those of $a$ and $b$ just before that step.
	We assume that these water levels satisfy $x_a > x_c > x_b$ without loss of generality.
	By~\ref{firsti}, we know that $a$ is red and $b$ is blue.
	By symmetry, we can furthermore assume that $c$ is blue too.\footnote{Otherwise, replace all water levels by their negatives and swap $a$ and $b$.}
	Then the same outcome can be obtained as follows:
	\begin{itemize}
		\item Equilibrate $a$ and $c$.
		\item Do a partial equilibration of $b$ and $c$ such that the water level in $c$ becomes $x_c$ again.
		\item Then equilibrate $a$ and $b$.
	\end{itemize}
	By construction, this new protocol has the same overall effect as the equilibration of $a$ and $b$ in $S'$.
	Similar to the argument for~\ref{firsti}, the second step is a weighted average of doing nothing and swapping the water levels of $b$ and $c$.
	Since $S'$ is optimal, neither of these strategies can be better than it, and therefore they must both have the same value as $S'$.
	Consider the case where we swap.
	In other words, let $S''$ be the strategy that coincides with $S'$ up to the equilibration of $a$ and $b$, which instead is replaced by equilibration of $a$ and $c$, and with $b$ and $c$ swapped in all subsequent steps;
	in particular, the third step above is another equilibration of $a$ and $c$, which is redundant.
	Then by what we have just argued, this $S''$ performs at least as good as $S'$, and clearly has the same number of steps.
	Since it has one fewer violation of~\ref{firstiii}, we are done.

	In conclusion, it is enough to restrict our attention to strategies which satisfy~\ref{firsti}--\ref{firstiii}, which we do from now on.
	In this case, we automatically have the following additional properties:
	\begin{enumerate}[resume]
		\item\label{secondi} No two tanks of the same color ever have the same water level, unless they both started at that water level.
		\item\label{secondii} Red tank $a$ and blue tank $b$ are equilibrated exactly once if $x_a > x_b$ initially, and not at all if $x_a < x_b$ initially.
	\end{enumerate}
	Indeed suppose that~\ref{secondi} fails, and consider the first step that caused such a situation.
	By \ref{firsti}, this step equilibrates a red and a blue tank, so one of the tanks must have already had that water level.
	But then the equilibration must have violated \ref{firstiii}.

	Concerning \ref{secondii}, suppose that we have $x_a > x_b$ initially, and that $a$ and $b$ are never equilibrated.
	Then in the final configuration we must have $x_a \le x_b$, or else we could equilibrate them and obtain a better strategy.
	Consider the step that first caused $x_a \le x_b$, and suppose for the sake of contradiction that it was not an equilibration of $a$ and $b$.
	Then clearly either $a$ was paired with a tank with less water than $b$, or $b$ was paired with a tank with more water than $a$, and either would violate~\ref{firstiii}.
	If $x_a < x_b$ initially, then it is enough to argue that this inequality remains valid throughout the protocol.
	Indeed if we consider any step which causes $x_a \ge x_b$ for the first time,
	then $a$ had to equilibrate with a blue tank $c$ with $x_a<x_b < x_c$, which contradicts~\ref{firsti}.

	In conclusion, \ref{firsti}--\ref{secondii} have been established as achievable properties of an optimal strategy.
	To understand the strategies of this type, it is intuitive to imagine all tanks arranged by initial water level as in \cref{pseudocodeg}: originally fuller tanks are to the left of originally emptier ones, and we break ties such that red tanks are to the \emph{right} of blue tanks of the same level.
	Then by~\ref{firstiii}, each step consists of an equilibration of a red tank with a blue tank immediately to the right of it.
	By the tie breaking rule, this requires swapping the two tanks.
	Thus the strategies under consideration are all of the form of \cref{pseudocodeg}.

	Conversely, we need to show that every strategy $S$ of the form of \cref{pseudocodeg} is optimal.
	If $S'$ is any strictly better strategy, then we can assume without loss of generality that it satisfies \ref{firsti}--\ref{secondii} as well.
	Therefore its steps also correspond exactly to the pairs $(a,b)$ consisting of a red tank $a$ which is to the left of a blue tank $b$ initially.
	Suppose that $S'$ starts by first equilibrating red tank $a$ with blue tank $b$.
	Then the pair $(a,b)$ also appears somewhere in $S$, and we can take it to be the first step without loss of generality, since all possible first steps commute.
	The proof is now complete by a straightforward induction argument.
\end{proof}

As a particular instance, \Cref{thm:nitrodon} states that the strategy of \cref{pseudocode1} is optimal.
We prove the quantitative result of \Cref{thm:optimaltransport} by a careful analysis of this strategy.

\begin{proof}[Proof of \Cref{thm:optimaltransport}]
	Consider an arbitrary labelling of all red tanks from $1$ to $n$, and likewise for all blue tanks.
	Our goal is to determine the total amount of water transferred from red to blue tanks by the strategy of \cref{pseudocode1}.
	
	Let $p_{i,j}$ be the water level of red tank $i$ after equilibrating with blue tank $j$.
	It will be convenient to define $p_{i,0} \coloneqq 1$ as the initial value of red tank $i$, and similarly $p_{0,j} \coloneqq 0$ for blue tank $j\geq 1$. Then, the $p_{i,j}$ satisfy the recurrence relation
	\[
		p_{i,j} = \frac{1}{2}\left(p_{i-1,j}+p_{i,j-1}\right) \qquad \text{for $i,j\geq 1$}.
	\]
	This relation holds because before red tank $i$ interacts with blue tank $j$, red tank $i$ has water level $p_{i,j-1}$ (the value it got after interacting with tank $j-1$, or the initial value for $j=1$) while tank $j$ has level $p_{i-1,j}$ (the value after interacting with red tank $i-1$, or the initial value for $i=1$).
	
	The amount of water left in red tanks at the end of the strategy is $\sum_{k=1}^n p_{k,n}$.  We aim to find a closed form expression for it.
	To this end, define $q_{i,j} \coloneqq \sum_{k=1}^i p_{k,j}$ for $i,j\geq 0$ and $q_{0,j} = q_{i,0} = 0$. For $i,j\geq 1$, we have the recurrence relation
	\[
		q_{i,j} = \frac{1}{2}\left(q_{i-1,j}+q_{i,j-1}\right)
	\]
	because of	
	\begin{align*}
		q_{i,j} = \sum_{k=1}^i p_{k,j} = \frac{1}{2} \sum_{k=1}^i \left(p_{k-1,j} + p_{k,j-1}\right) = \frac{1}{2} \left(\sum_{k=1}^{i-1} p_{k,j} + \sum_{k=1}^i p_{k,j-1}\right) = \frac{1}{2}\left(q_{i-1,j}+q_{i,j-1}\right),
	\end{align*}
	and the initial conditions $q_{0,j} = 0$ and $q_{i,0} = i$.
	The explicit solution of this recurrence is given by
	\[
		q_{i,j} = \frac{1}{2^{i+j}} \sum_{k=1}^j q_{k,i+j-k} 
	\]

	To evaluate $q_{i,j}$ we resort to generating functions.
	Thus define
	\[
		Q(x,y) \coloneqq \sum_{i,j=1}^{\infty} q_{i,j} x^i y^j
	\]
	for $x,y\in \mathbb{C}$.
	Since $0\leq q_{i,j} \leq (i+j)/2$ (which follows from a simple induction argument), the defining power series converges for all $(x,y)\in \mathbb{C}^2$ with $|x|< 1$ and $|y|< 1$.
	By the recurrence relation, for $|x|< 1$ and $|y|< 1$ we have
	\[
		Q(x,y) = \frac{1}{2} \sum_{i,j=1}^{\infty} (q_{i-1,j}+q_{i,j-1}) x^i y^j = \frac{x+y}{2} \, Q(x,y) + \frac{x}{2(1-x)^2},
	\]
	where the second term comes from $\sum_{i=1}^\infty q_{i,0} x^i = \sum_{i=1}^\infty i x^i = x/(1-x)^2$.
	Therefore
	\[
		Q(x,y) = \frac{x}{(1-x)^2 (2 - x - y)}.
	\]
	Since we are interested in $q_{n,n}$, we next extract the diagonal of $Q(x,y)$, i.e.
	\[
		\Delta(z) = \sum_{n\geq 0} q_{n,n}z^n.
	\]
	A classic approach to obtain this diagonal is to use the residue theorem (see e.g. Theorem 2.32 in \cite{PW2024Analytic}). Fix $z\in \mathbb{C}$ and consider the complex-valued function
	\[
		Q(t,z/t) = \frac{zt}{(1-t)^2 \left(2t - z - t^2 \right)}.
	\]
	Since $Q$ converges around $(0,0) \in \mathbb{C}^2$, then for small enough $z > 0$,we have that the Laurent series expansion of $Q(t,z/t)$ in $t$ converges in an annulus $A(z)$ around $0$.
	Thus the constant term $[t^0]$ (which is a function of $z$) is the diagonal $\Delta(z)$. Let $C$ be a circle in $A(z)$ such that the origin is inside $C$, then Cauchy's integral formula yields
	\[
		\Delta(z) = [t^0]Q(t,z/t) = \frac{1}{2\pi \ii}\oint_{C} \frac{Q(t,z/t)}{t} \, \mathrm{d}t = \frac{1}{2\pi \ii}\oint_C \frac{z}{(1-t)^2 (2t - z - t^2)} \, \mathrm{d}t.
	\]
	By choosing $z$ small enough, we can choose the circle $C$ as close to $0$ as desired, so that we can apply the residue theorem ignoring large poles.
	The poles are $1$ and $1 \pm \sqrt{1-z}$, but only $1-\sqrt{1-z}$ is inside the circle $C$ for small $z$. 
	Some computation shows that the residue at $1-\sqrt{1-z}$ is
	\[
		\Delta(z) = \frac{z}{2 (1-z)^{3/2}} 
	\]
	for all $z$ in a region around $0$.
	Using the standard negative binomial series expansion gives
	\[
		\Delta(z) = \frac{1}{2} \sum_{n=1}^{\infty} \binom{n - \frac{1}{2}}{n - 1} z^n.
	\]
	Hence we can extract the desired coefficient
	\[
		q_{n,n} = \frac{1}{2} \binom{n - \frac{1}{2}}{n - 1} = \frac{\Gamma(n+\frac{1}{2})}{\sqrt{\pi} \Gamma(n)} = \frac{2}{4^n} \cdot \frac{\Gamma(2n)}{\Gamma(n)^2} 
		= \frac{n}{4^n} \cdot \frac{\Gamma(2n + 1)}{\Gamma(n + 1)^2}
		= \frac{n}{4^n} \binom{2n}{n},
	\]
	where the third step uses the Legendre duplication formula.
	This is the desired closed form expression for the amount of water left in the red tanks at the end of the strategy.

	The asymptotic expansion of the central binomial coefficient 
	\[
	\binom{2n}{n} = \frac{4^n}{\sqrt{\pi n}} \left( 1 - \frac{1}{8 n} + O(n^{-2})  \right),
	\]
	is standard, and follows e.g.~from Stirling's formula.
	This implies that $\frac{n}{4^n} \binom{2n}{n} = \sqrt{\frac{n}{\pi}} + O(n^{-{1/2}})$, completing the proof.
\end{proof}

\section{Proof of \texorpdfstring{\cref{dynamics}}{Theorem 1.2}}

\begin{lem}
	If $S,T : X \to X$ are measure-preserving transformations which satisfy the claim of \cref{dynamics}, then so does $T \circ S$.
\end{lem}

\begin{proof}
If $\mathcal{G} \subseteq \Sigma$ is any sub-$\sigma$-algebra, then
\[
	\mathcal{G}' \coloneqq \{ T(A) \: : \: A \in \mathcal{G} \}
\]
is a sub-$\sigma$-algebra as well. 
With the respective conditional expectation operators on $L^\infty(X)$ denoted $\E$ and $\E'$, we have
\[
	\E'(f) \circ T = \E(f \circ T)
\]
for all $f \in L^\infty(X)$.\footnote{This is because for every $\mathcal{G}'$-measurable $g$, we have
	\[
		\int g(x) \, \E'(f)(T(x)) \, \mu(dx) = \int g(T^{-1}(x')) \, \E'(f)(x') \, \mu(dx') = \int g(T^{-1}(x')) \, f(x') \, \mu(dx') = \int g(x) \, f(T(x)) \, \mu(dx),
	\]
	using the substitution $x = T^{-1}(x')$ and the fact that $g \circ T^{-1}$ is $\mathcal{G}'$-measurable.}
If we now assume
\begin{align*}
	\left\| f \circ T - \E_m \cdots \E_1 f \right\|_\infty & < \eps \\[2pt]
	\left\| (f \circ T) \circ S - \E_{m + \ell} \cdots \E_{m + 1} (f \circ T) \right\|_\infty & < \eps,
\end{align*}
then we can estimate
\begin{align*}
	\big\| f \circ T \circ S - \E'_{m + \ell} \cdots \E'_{m + 1} & \E_m \cdots \E_1 f \big\|_\infty \\[2pt]
		& \le \left\| (f \circ T) \circ S - \E_{m + \ell} \cdots \E_{m + 1} (f \circ T) \right\|_\infty \\[2pt]
		& \qquad + \left\| \E'_{m + \ell} \cdots \E'_{m + 1} f \circ T - \E'_{m + \ell} \cdots \E'_{m + 1} \E_m \cdots \E_1 f \right\|_\infty \\[2pt]
		& < \eps + \eps,
\end{align*}
where the second step also uses the fact that taking conditional expectations is contractive in the supremum norm.
\end{proof}

It was shown by Ryzhikov that every measure-preserving automorphism on a standard Lebesgue space can be written as a composition of three involutive ones~\cite{Ryzhikov}. 
Therefore for the proof of \cref{dynamics}, we can assume without loss of generality that $T$ is an involution.
Then by making a choice of measurable selection for the orbits of order $2$, we can decompose $X$ into a disjoint union
\[
	X = F \sqcup A \sqcup B,
\]
where $F$ is the set of fixed points of $T$ and $B = T(A)$.
We focus on the case $F = \emptyset$ from now on, since the following construction can trivially be extended from $A \sqcup B$ to all of $X$.

In the remaining case $X = A \sqcup B$, assume $\|f_i\| \le 1$ for all $i$ without loss of generality, and consider the $f_i$ together as a single function $f : X \to [-1,1]^n$.
Choosing a partition of the unit cube $[-1,1]^n$ into subsets $S_1, \dots, S_\ell$ of diameter $< \eps$ induces a partition of $X$ into the sets
\begin{align*}
	A_{i,j} & \coloneqq f^{-1}(S_i) \cap T(f^{-1}(S_j)) \\
	B_{i,j} & \coloneqq T(f^{-1}(S_i)) \cap f^{-1}(S_j)
\end{align*}
for $i, j = 1, \ldots, \ell$, so that $B_{i,j} = T(A_{j,i})$.
By definition, $f$ varies by less than $\eps$ on each of these pieces.
By first applying a conditional expectation which averages $f$ on these pieces, we can assume without loss of generality that $f$ is actually constant on each piece, since this changes the values of $f$ by at most $\eps$.

For fixed $i$ and $j$, we choose a further partition of $A_{i,j}$ into many subsets of equal measure and use $T$ to transport this partition into one of $B_{i,j}$.
If we now identify these subsets with the water tanks of \cref{tanks}, then we are exactly in a situation where \cref{tanks} applies, where an equilibration step clearly corresponds to a conditional expectation.
Hence we can find a sequence of conditional expectation operators which takes $f|_{A_{i,j} \cup B_{i,j}}$ to $(f \circ T)|_{A_{i,j} \cup B_{i,j}}$ with arbitrary accuracy and does not change $f$ anywhere else.
Doing this for all $i$ and $j$, and noting that these operations commute with each other as $i$ and $j$ vary, finishes the proof of \cref{dynamics}.

\section{The finite case and Dalton transfers}

In this final section, we comment on the connection between \cref{tanks} and the theory of majorization, and also discuss some related open problems.
We recall the definition first.

\begin{defn}[{e.g.~\cite{MOA}}]
	\label{majorization}
	For $x, y \in \R^n$, one says that $x$ \newterm{majorizes} $y$, written $x \succeq y$, if there is a doubly stochastic matrix $D$ such that $y = D x$.
\end{defn}

Among other equivalent formulations, the majorization relation $\succeq$ can also be characterized as the transitive closure of the following of its instances:
\begin{itemize}
	\item For every permutation matrix $\pi$ and $x \in \R^n$, we have $x \succeq \pi x$.
	\item If $x, y \in \R^n$ differ only in coordinates $i$ and $j$, and for these are such that
		\begin{equation}
			\label{robin_hood}
			x_i \le y_i \le y_j \le x_j, \qquad x_i + x_j = y_i + y_j,
		\end{equation}
		then $x \succeq y$. 
\end{itemize}
If one takes the vectors $x, y \in \R^n$ to model distributions of wealth or income of $n$ individuals, then~\eqref{robin_hood} amounts to a redistribution of some wealth from richer individual $j$ to poorer individual $i$, but only so much that $j$ remains at least as wealthy as $i$, as expressed by the inequality $y_i \le y_j$.
Such an operation is known as a \newterm{Dalton transfer}, and the wealth interpretation lies behind the applications of majorization to social welfare theory and measures of economic inequality~\cite{DSS,TW}.
Majorization is important also in other areas, for example in quantum information theory in the context of entanglement transformations~\cite{nielsen}.

We can now think of the water tank problem as asking what happens if one forbids the permutations and \emph{only} allows Dalton transfers as in~\eqref{robin_hood}.

\begin{defn}
	For $x, y \in \R^n$, we write $x \succeq_D y$ if there is a finite sequence of Dalton transfers that transforms $x$ into $y$.
\end{defn}

Since the relation $\succeq_D$ is the transitive closure of the Dalton transfers by definition, it is also a partial order relation on $\R^n$.
How do we decide whether $x \succeq_D y$ holds for given $x, y \in \R^n$?

As far as we know, the partial ordering $\succeq_D$ and the problem of its characterization was first studied by Zylka with heat transfer in mind~\cite{zylka}.
Subsequently, Thon and Wallace have expanded on Zylka's work and emphasized the connection to social welfare theory~\cite{TW}.
To state the characterization of $\succeq_D$ developed by Thon and Wallace based on Zylka's results, let us say that a \newterm{full Dalton transfer} is a Dalton transfer as in~\eqref{robin_hood} with $y_i = y_j = (x_i + x_j)/2$, or equivalently an equilibration step as in the previous sections.

\begin{thm}[{\cite{zylka,TW}}]
	\label{dalton_characterization}
	For $x, y \in \R^n$ we have $x \succeq_D y$ if and only if there is a sequence of full Dalton transfers
	\begin{equation}
		\label{z_sequence}
		z^{(0)} \rightsquigarrow z^{(1)} \rightsquigarrow \cdots \rightsquigarrow z^{(k)}
	\end{equation}
	with
	\[
		z^{(0)} = x, \qquad z^{(k)} \succeq y,
	\]
	and such that:
	\begin{enumerate}
		\item No two transfers happen between the same pair of individuals.
		\item Each transfer $z^{(i-1)} \rightsquigarrow z^{(i)}$ is between two individuals for which their wealth relation at time $i-1$ satisfies that
			\begin{itemize}
				\item it is opposite to that in $y$, and
				\item there is no other individual of strictly intermediate wealth at that time.
			\end{itemize}
	\end{enumerate}
\end{thm}

This result facilitates an algorithmic check of whether $x \succeq_D y$ holds, since the number of possible sequences~\eqref{z_sequence} is finite and the majorization relation $z^{(m)} \succeq y$ can be decided by checking a finite number of inequalities~\cite{MOA}.
There is an elegant recursive implementation of this algorithm~\cite[p.~464]{TW}.
One can use it to show that the set of points reachable from $x = (13,4,1)$ in $\R^3$ is not convex~\cite[Fig.~1]{TW}.

In quantum thermodynamics, recent work by Lostaglio and Korzekwa has considered a more general version of $\succeq_D$ which they call \emph{continuous thermomajorization}, for which they have also proven a variant of \cref{dalton_characterization}~\cite[Theorem~4]{LK}.

\printbibliography

\end{document}